\documentclass[12pt]{extarticle}
\usepackage{comment}
\usepackage{geometry}
\usepackage{tabularx}
\def\version{30 December 2024}

\usepackage{ifpdf}

\newcommand{\ac}[1]{\noindent\textcolor{red}
{{\rm [\![}\mbox{\sc{AC}$\blacktriangleright\!\!\blacktriangleright$}: {#1}{\rm ]\!]}}}

\renewcommand{\ac}[1]{}
\usepackage{mathtools}
\usepackage{latexsym,epsfig,bm}
\usepackage{upgreek}
\usepackage{mathrsfs}
\usepackage{amssymb}

\usepackage{upref}
\usepackage{esint}

\nonstopmode

\usepackage[english]{babel}

\usepackage{ulem} 
\usepackage[usenames,dvipsnames]{color}
\usepackage{graphicx}

\definecolor{MyDarkBlue}{rgb}{0,0.08,0.45}
\definecolor{MyDarkGreen}{rgb}{0,0.7,0}
\definecolor{MyGreen}{rgb}{0,0.7,0.1}
\definecolor{Pomegranade}{rgb}{0.6,0.1,0.15}
\definecolor{purple}{rgb}{0.6,0.1,0.15}
\usepackage[backref=none]{hyperref}
\hypersetup{pdfborder={0 0 0},
  colorlinks,
  urlcolor={MyDarkBlue},
  linkcolor={MyDarkBlue},
  citecolor={MyDarkBlue},
  breaklinks=true}
\providecommand{\url}[1]{\small\textcolor{blue}{#1}}
\usepackage{doi}
\providecommand{\eprint}[1]{}
\renewcommand{\eprint}[1]{arXiv:\href{http://arxiv.org/abs/#1}{#1}}

\providecommand{\eqref}[1]{\mathrm{ (\ref{#1})}}

\newcommand{\unity}{\textrm{{\usefont{U}{fplmbb}{m}{n}1}}}

\newcommand{\fra}[2]{{#1}/{#2}}

\providecommand{\longhookrightarrow}{\lhook\joinrel\longrightarrow}

\newcommand{\dom}{\mathfrak{D}}
\newcommand{\range}{\mathfrak{R}}
\newcommand{\ran}{\mathfrak{R}}

\newcommand{\jj}{\mathrm{i}}

\newcommand{\scrB}{\mathscr{B}}
\newcommand{\scrC}{\mathscr{C}}
\newcommand{\scrD}{\mathscr{D}}

\newcommand{\calB}{\mathcal{B}}

\newcommand{\calK}{\mathcal{K}}

\newcommand{\calR}{\mathcal{R}}
\newcommand{\calS}{\mathcal{S}}

\newcommand{\bfE}{\mathbf{E}}
\newcommand{\bfF}{\mathbf{F}}

\newcommand{\doma}{\dom(\hatA)}

\newcommand{\bfX}{\mathbf{X}}

\newcommand{\hatA}{\hat{A}}

\newcommand{\notyet}[1]{{}}

\newcommand{\tr}{\mathop{\mathrm{tr}}}
\newcommand{\rank}{\mathop\mathrm{ rank}}

\newcommand{\p}{\partial}
\newcommand{\at}[1]{\!\!\upharpoonright\sb{{#1}}}

\newcommand{\R}{\mathbb{R}}
\newcommand{\C}{\mathbb{C}}

\newcommand{\N}{\mathbb{N}}
\newcommand{\Abs}[1]{\left\vert#1\right\vert}
\newcommand{\abs}[1]{\vert #1 \vert}

\newcommand{\norm}[1]{\Vert #1 \Vert}

\newcommand{\sothat}{\,\,{\rm :}\ \ }

\newcommand{\wlim}{\mathop{\mbox{\rm w-lim}}}

\DeclareMathSymbol{\varGamma}{\mathord}{letters}{"00}
\DeclareMathSymbol{\varDelta}{\mathord}{letters}{"01}
\DeclareMathSymbol{\varTheta}{\mathord}{letters}{"02}
\DeclareMathSymbol{\varLambda}{\mathord}{letters}{"03}
\DeclareMathSymbol{\varXi}{\mathord}{letters}{"04}
\DeclareMathSymbol{\varPi}{\mathord}{letters}{"05}
\DeclareMathSymbol{\varSigma}{\mathord}{letters}{"06}
\DeclareMathSymbol{\varUpsilon}{\mathord}{letters}{"07}
\DeclareMathSymbol{\varPhi}{\mathord}{letters}{"08}
\DeclareMathSymbol{\varPsi}{\mathord}{letters}{"09}
\DeclareMathSymbol{\varOmega}{\mathord}{letters}{"0A}

\usepackage{amsthm,xpatch}
\makeatletter
\xpatchcmd{\@thm}{.}{\,}{}{}
\makeatother

\theoremstyle{plain}
\newtheorem{lemma}{Lemma}[section]
\newtheorem{theorem}[lemma]{Theorem}
\newtheorem{corollary}[lemma]{Corollary}

\theoremstyle{definition}
\newtheorem{definition}[lemma]{Definition}

\newtheorem{assumption}{Assumption}

\theoremstyle{remark}

\newtheorem{remark}{Remark}[section]

\newcounter{step}

\makeatletter\@addtoreset{equation}{section}
\makeatletter\@addtoreset{lemma}{section}
\makeatother

\renewcommand{\Re}{\mathop{\rm{R\hskip -1pt e}}\nolimits}
\renewcommand{\Im}{\mathop{\rm{I\hskip -1pt m}}\nolimits}

\renewenvironment{abstract}
  {\list{}{\listparindent 0.0cm%
  \baselineskip 0pt
  \setlength{\leftmargin}{1.5cm}
  \setlength{\rightmargin}{1cm}
  }%
  \item\relax \hskip -10pt {\sc Abstract.}\ \footnotesize}
  {\endlist}

\newcommand{\shortto}[1][3pt]{\mathrel{%
   \hbox{\hskip 1pt\rule[\dimexpr\fontdimen22\textfont2-.2pt\relax]{#1}{.4pt}}%
   \mkern-6mu\hbox{\usefont{U}{lasy}{m}{n}\symbol{41}}}}

\begin{document}

\title{
Virtual levels, virtual states,
and the limiting absorption principle
for higher order differential operators
in 1D
}

\author{
{\sc Andrew Comech}
\\
{\it\small Texas A\&M University, College Station, Texas, USA
}
\\[2ex]
{\sc Hatice Pekmez}
\\
{\it\small Texas A\&M University, College Station, Texas, USA
}
}

\date{\version}

\maketitle



\begin{abstract}
We consider the resolvent estimates
and properties of virtual states
of the higher order derivatives
in one dimension,
focusing on
Schr\"odinger-type
operators of degree $N=3$
(the method applies to higher orders).
The derivation is based on
the construction of the Jost solution
for higher order differential operators
and on restricting the resolvent
onto subspaces of finite codimension.
\end{abstract}

\section{Introduction}

Let us recall the general picture
\cite{virtual-levels,virtual-levels-review}.
We consider a closed operator
$A\in\scrC(\bfX)$
in a complex Banach space $\bfX$.
The norm of its resolvent
$(A-z I_\bfX)^{-1}$, of course,
becomes unboundedly large
when $z$ approaches the essential spectrum
of $A$.
Yet the resolvent may have a limit
as an operator in some auxiliary spaces;
then we say that at a particular point $z_0$
of the essential spectrum
the resolvent satisfies
the \emph{limiting absorption principle} (LAP),
as a mapping in these spaces.
This idea goes back to
\cite{ignatowsky1905reflexion,smirnov1941course,sveshnikov1950radiation}
and took up its present form in
\cite{agmon1970spectral}.

When we add to $A$
a relatively compact perturbation $B$,
the resolvent
$R_B(z)=(A+B-z I_\bfX)^{-1}$
either satisfies
\emph{the same}
LAP (at the same point, in the same spaces)...
or not.
In the latter case,
the resolvent is not uniformly bounded
near $z_0$ as a mapping in these spaces
\cite{virtual-levels};
we say that the resulting operator
$A+B$
has a \emph{virtual level} at a given point.

In other words, virtual levels
correspond to particular singularities
of the resolvent at the essential spectrum.
This idea goes back to
E.\,Wigner \cite{wigner1933streuung}
and
H.\,Bethe and R.\,Peierls \cite{bethe1935scattering}
and was further addressed
by Birman \cite{birman1961spectrum},
Faddeev \cite{faddeev1963mathematical},
Simon \cite{simon1973resonances,simon1976bound},
Vainberg \cite{vainberg1968analytical,vainberg1975short},
Yafaev \cite{yafaev1974theory,yafaev1975virtual},
Rauch \cite{rauch1978local},
and Jensen and Kato \cite{jensen1979spectral},
with the focus on
Schr\"odinger operators
in three
dimensions.
Uniform resolvent estimates for
Schr\"odinger operators in higher dimensions appeared in
\cite{kenig1987uniform},
\cite{frank2011eigenvalue},
\cite{frank2017eigenvalue}, \cite{gutierrez2004nontrivial},
\cite{bouclet2018uniform}, \cite{ren2018endpoint}, \cite{mizutani2019eigenvalue}, \cite{kwon2020sharp}.
For
the Laplacian in $\R^d$, $d\ge 3$,
the $L^p\to L^{p'}$ resolvent estimates
were proved in \cite{kenig1987uniform}.

Dimensions $d=1$ and $d=2$
are exceptional,
in the sense that the free Laplace operator
has a virtual level at the threshold $z_0=0$
and does not satisfy LAP
uniformly in an open neighborhood of the threshold.
The Schr\"odinger operators
in one and two dimensions have been covered in
\cite{bolle1985complete,bolle1987scattering}
and in 
\cite{bolle1988threshold}
and then by Jensen and Nenciu in
\cite{jensen2001unified,jensen2004erratum}.
To derive the optimal resolvent estimates
near the threshold,
one needs to either consider a particular
perturbation
of the Laplacian which destroys the virtual level,
or to restrict the Laplacian onto a space
of finite codimension;
this program was completed
in \cite{virtual-levels}
for the Schr\"odinger operators
(with complex-valued potentials),
giving optimal resolvent estimates
(when there is no virtual level at the threshold)
and optimal properties of the virtual states
(when there is a virtual level at the threshold).
For related results on properties of virtual states
for selfadjoint
Schr\"odinger operators in dimensions $d\le 2$,
see
\cite[Theorem 2.3]{barth2021absence}.

As an illustration, let us consider the
Schr\"odinger operator
with (complex-valued) compactly supported potential
(we closely follow the exposition
in \cite{virtual-levels,virtual-levels-review}):
\begin{align}
A=-\p_x^2+V\in\scrC(L^2(\R,\C)),
\qquad
V\in C_{\mathrm{comp}}(\R,\C),
\qquad
\dom(A)=H^2(\R,\C).
\end{align}
For such $V$,
one has $\sigma\sb{\mathrm{ess}}(-\p_x^2+V)=[0,+\infty)$.
For $z\in\C\setminus[0,+\infty)$
the resolvent
$R_V(z)=(-\p_x^2+V-z)^{-1}$
can be constructed in terms of the Jost solutions.
Assume that
$z=\zeta^2$ with $\zeta\in\C_{+}$.
Then the Jost solutions
$\theta_\pm(x,\zeta)$
can be characterized by
\[
\begin{cases}
\theta_{+}(x,\zeta)\approx
e^{\jj\zeta x},
\quad
&
x\to+\infty;
\\
\theta_{-}(x,\zeta)\approx
e^{-\jj\zeta x},
\quad
&
x\to-\infty.
\end{cases}
\]
The Jost solutions
$\theta_\pm$ also have limits
$\theta_\pm(x,\zeta_0)$
as $\zeta\to\zeta_0\in\R$.
Since we assume that
$V$ is compactly supported,
the above relations turn
into equalities for $\abs{x}$ large enough.
The corresponding  value of $z=\zeta^2$
is an eigenvalue
if $\theta_+(x,\zeta)$
and $\theta_-(x,\zeta)$
are linearly dependent,
so that their Wronskian
\[
W[\theta_{+},\theta_{-}](\zeta)
=\theta_{+}(x,\zeta)\p_x\theta_{-}(x,\zeta)-
\p_x\theta_{+}(x,\zeta)\theta_{-}(x,\zeta)
\]
is equal to zero
(as usual, $W[\theta_{+},\theta_{-}]$
is $x$-independent).
If instead
the Wronskian does not vanish at the corresponding
value of $\zeta$,
then the resolvent
$R_V=(A-z I)^{-1}$,
$z=\zeta^2$,
of $A=-\p_x^2+V$
is represented by the integral operator
with integral kernel
\begin{align}\label{def-gv}
G_V(x,y;\zeta)=
\frac{1}{W[\theta_{+},\theta_{-}](y,\zeta)}
\begin{cases}
\theta_{+}(x,\zeta)\theta_{-}(y,\zeta),&x\ge y,
\\
\theta_{-}(x,\zeta)\theta_{+}(y,\zeta),&x\le y.
\end{cases}
\end{align}
The above resolvent has a pointwise
limit (in $x,\,y$)
as $\zeta\to 0$ as long as
$W[\theta_+,\theta_-](\zeta)$
does not vanish at $\zeta_0=0$;
then one says that
(the resolvent of) $A$
satisfies the limiting absorption principle at $z_0=0$.
If $W[\theta_+,\theta_-](\zeta)$ vanishes
at $\zeta_0=0$,
one says that $A$ has a virtual level at $z_0=0$.
In this case,
$\theta_-(x,0)$ and $\theta_+(x,0)$
are linearly dependent,
thus the equation
$(A-z_0)u=0$ has a nontrivial bounded solution.
We note that the equation $(A-z_0)u=0$
always has nontrivial solutions
which grow linearly at infinity;
$z_0=0$ is a virtual level if and only if
there is a bounded nontrivial solution.

What about the resolvent estimates
in the limit $\zeta\to 0$
(if $z_0=0$ is \emph{not} a virtual level,
that is,
when $W[\theta_+,\theta_{-}](0)\ne 0$)?
Since there are no nontrivial bounded solutions to
$Au=0$,
we conclude that $\theta_+(x,0)\approx 1$  for large positive $x$
grows linearly as $x\to-\infty$,
while $\theta_-$ grows linearly as $x\to+\infty$.
Then \eqref{def-gv}
shows that
\begin{align}\label{agv}
\abs{G_V(x,y;0)}
\le C\min(\langle x\rangle,\langle y\rangle),
\end{align}
with some $C>0$;
in fact, this estimate
holds for $\abs{G_V(x,y;\zeta)}$
not only at $\zeta=0$, but it also holds
uniformly in
$\zeta\in\C_{+}\cap\mathbb{D}_\delta$
with $\delta>0$ small enough.
Now one can easily derive the corresponding estimates,
showing that the resolvent is uniformly bounded
as a mapping
\begin{align}\label{ll2}
L^2_s(\R,\C)\to L^2_{-s'}(\R,\C),
\qquad
s,\,s'>1/2,
\quad
s+s'\ge 2.
\end{align}
For details, we refer the reader to
\cite[Section 3]{virtual-levels}.

\medskip

The above picture looks very simple
in the case of
the Schr\"odinger equation in one dimension
as long as $V$ is compactly supported.
In fact, there is not too much change
if $V$ is not compactly supported
yet satisfies
$\langle x\rangle^{2+0}V\in L^\infty$,
or even a weaker assumption
$\langle x\rangle V\in L^1$
which is
sufficient for the construction of Jost solutions.
We will show
how the case of higher derivatives,
\begin{align}
A=(-\jj\p_x)^N+V\in\scrC(L^2(\R,\C)),
\qquad
V\in C_{\mathrm{comp}}(\R,\C),
\qquad
\dom(A)=H^N(\R,\C),
\end{align}
can also be discussed along the same lines.
The questions that we are going to answer in
this article are:
\begin{enumerate}
\item
If $z_0=0$ is not a virtual level,
then what are the optimal
resolvent estimates? That is,
in which spaces does the limiting absorption
principle hold?
\item
If $z_0=0$ is a virtual level,
what are the properties
of corresponding virtual states?
\end{enumerate}

We focus on the case $N=3$,
completely answering the above questions
in Theorem~\ref{theorem-2} below.
In this case, for each $\zeta$
with $0\le\arg(\zeta)\le\pi/3$,
$\zeta^3=z\in\overline{\C_{+}}$,
there are three Jost solutions
to the equation
\begin{align}
\big((-\jj\p_x)^3+V\big)u=\zeta^3 u,
\end{align}
two decaying (or bounded)
in one direction of $x$
and one in the other.
Denote $\alpha=e^{2\pi\jj/3}$;
then
$\theta_1\approx e^{\jj\zeta x}$,
$\theta_2\approx e^{\jj\alpha\zeta x}$
for $x\to+\infty$
remain bounded
for $x\ge 0$,
while
$\gamma$,
which behaves as $e^{\jj\alpha^2\zeta x}$
for $x\to-\infty$,
remains bounded  for negative $x$.
For each fixed $y$,
one can use the linear combinations
of these Jost solutions,
\begin{align}\label{def-ux}
u(x)=\begin{cases}
c_1\theta_1(x,\zeta)+c_2\theta_2(x,\zeta),&
x\ge y;
\\
c\gamma(x,\zeta),&y\le x,
\end{cases}
\end{align}
to construct a function
that would be $C^1$ on $\R$
but such that its second derivative
would have a jump at $y$;
this is the expression for the resolvent
$(A-z I)^{-1}$
at $z=\zeta^3$.
This would fail if at a particular $\zeta$
the functions
$\theta_1$, $\theta_2$, and $\gamma$
are linearly dependent;
this means that 
\eqref{def-ux}
has zero jump of the second derivative at $x=y$,
thus
$u(x)$ is an $L^2$-eigenfunction
corresponding to $z=\zeta^3$.

In the limit $\zeta\to 0$,
both
$\theta_1(x,\zeta)$ and $\theta_2(x,\zeta)$
converge (pointwise) to the same function $\theta_0(x)$
-- solution to $\big((-\jj\p_x)^3+V\big)u=0$ --
which equals $1$ for $x\gg 1$;
the function
$\gamma(x,\zeta)$
converges to
$\gamma_0(x)$,
another solution to $\big((-\jj\p_x)^3+V\big)u=0$
which equals $1$ for $x\ll -1$.
There is no virtual level
at $z=0$ if
$\theta_0(x)$
and
$\gamma_0(x)$
grow quadratically as $x\to-\infty$
and $x\to+\infty$, respectively.

What about the estimates in the case when there is
no virtual level at $z=0$?
Now \eqref{agv} takes the form
\begin{align}\label{agv-n}
\abs{G_V(x,y;0)}
\le C\min(\langle x\rangle,\langle y\rangle)^{N-2};
\end{align}
\eqref{ll2} takes the form
\begin{align}\label{lln}
L^2_s(\R,\C)\to L^2_{-s'}(\R,\C),
\qquad
s,\,s'>N-3/2,
\quad
s+s'\ge N.
\end{align}
Of course, for $N\ge 3$, the
condition $s+s'\ge N$ becomes redundant.
The estimate
\eqref{lln}
is our main result:

\begin{theorem}
\label{theorem-1}
Let $N=3$.
Let $\varOmega=\C_{+}$.
Consider
$A=(-\jj\p_x)^N$ in $L^2(\R,\C)$,
with domain $\dom(A)=H^N(\R,\C)$.
Let
$s,\,s'>N-3/2$
 and let
$B:\,L^2_{-s'}(\R,\C)\to L^2_{s}(\R,\C)$
be $A$-compact.\footnote{In the sense that
  the set $\ran(B\at{\mathbb{B}_1(\hat{A})})$
  is precompact in $L^2_s(\R,\C)$,
  with $\hat{A}\in\scrC(\bfF)$
  a closed extension of $A$ onto $\bfF=L^2_{-s'}(\R,\C)$
  and
  $\mathbb{B}_1=\{f\in\dom(\hat{A})\sothat
  \norm{f}_\bfF+\norm{\hat{A}f}_{\bfF}\}$;
  for more details,
  see \cite[Definition A.1]{virtual-levels}.}
Then
\begin{itemize}
\item
either the resolvent of
$A+B$ satisfies the limiting absorption
principle at $z_0=0$
with respect to $L^2_s,\,L^2_{-s'},\,\varOmega$,
so that the resolvent $R_B(z)=(A+B-z I)^{-1}$
converges in the uniform operator topology
of $\scrB\big(L^2_s,L^2_{-s'}\big)$
as $z\to z_0$, $z\in\varOmega$,
\item
or
there is a virtual state
$\Psi\in L^2_{-s'}$,
$(A+B-z_0)\Psi=0$,
$\Psi\in\ran\big(
(A+B-z_0 I)^{-1}_{L^2_s,L^2_{-s'},\varOmega}\big)$.
\end{itemize}
\end{theorem}

In the case when
the perturbation $B$
is represented by the function
$V\in C_{\mathrm{comp}}(\R,\C)$,
we have the following result,
which gives the characterization
of the virtual state:

\begin{theorem}
\label{theorem-2}
Let $N=3$.
Consider
$A=(-\jj\p_x)^N$ in $L^2(\R,\C)$,
with domain $\dom(A)=H^N(\R,\C)$.
Let
$s,\,s'>N-3/2$
and let $V$ be the operator
of multiplication by $V\in C_{\mathrm{comp}}(\R,\C)$.
Then

\begin{itemize}
\item
either the resolvent
\[
(A+V-z I)^{-1}:\,
L^2_s(\R,\C)
\to
L^2_{-s'}(\R,\C),
\qquad
z\in\varOmega
\]
converges in the uniform operator topology
of $\scrB\big(L^2_s,L^2_{-s'}\big)$
as $z\to z_0$, $z\in\varOmega$,
\item
or there is a solution
$\langle x\rangle^{N-2}\Psi\in L^\infty(\R,\C)$
to $\big((-\jj\p_x)^3+V\big)\Psi=0$;
moreover,
if $N$ is odd,
this solution is bounded
by $\langle x\rangle^{N-3}$
for $x\le 0$.
\end{itemize}
\end{theorem}

While we only give a proof of this result
for $N=3$, we expect that it holds for all $N\ge 3$.

\begin{remark}
In other words,
we are saying that
if $z_0=0$ is a virtual level
relative to
$L^2_s,\,L^2_{-s'},\,\C_{+}$,
then
the corresponding virtual state
grows at most linearly as $x\to+\infty$
and remains uniformly bounded
as $x\to-\infty$;
similarly,
if $z_0$ is a virtual level
relative to
$L^2_s,\,L^2_{-s'},\,\C_{-}$,
then
the corresponding virtual state
grows at most linearly as $x\to-\infty$
and remains uniformly bounded
as $x\to+\infty$.
If there is a
nontrivial solution
to $\big((-\jj\p_x)^3+V\big)\psi=0$
such that $\langle x\rangle^{N-3}\Psi\in L^\infty(\R)$,
then $z_0=0$
is a virtual level both relative
to $\C_{+}$ and relative to $\C_{-}$;
arbitrarily small potentials
can produce eigenvalues near $z_0=0$
either in $\C_{+}$ or in $\C_{-}$.

This absence of symmetry
takes place when $N$ is odd;
one can see that in this case
Theorem~\ref{theorem-2} applies to
virtual levels at $z_0=0$
relative to $\varOmega=\C_{-}$
by simultaneously
changing the sign of $z$
and the sign of $x$.
\end{remark}

\bigskip

Here is the structure of the article.
We remind the general theory in
Section~\ref{sect-main-results}.
Properties of Jost solutions
of higher order operators
are given in Section~\ref{sect-jost-N}.
Construction of the resolvent in terms of
Jost solutions is given in
Section~\ref{sect-resolvent}.
The limiting absorption
principle at the threshold point
(in the absence of a virtual level)
-- that is, uniform resolvent estimates --
are derived in Section~\ref{sect-finite-codimension}.
While Theorem~\ref{theorem-1}
is just a reformulation of the general theory
\cite{virtual-levels} in our case,
Theorem~\ref{theorem-2} is our main result;
we prove it
at the end of Section~\ref{sect-finite-codimension}.

\section{Virtual levels and virtual states in Banach spaces}
\label{sect-main-results}

Here we remind the general theory from
\cite{virtual-levels}.
Let $\bfX$ be an infinite-dimensional complex Banach space
and let
$A\in\scrC(\bfX)$
be a closed linear operator
with dense domain $\dom(A)\subset\bfX$.
We say that $\lambda$ is from the point spectrum
$\sigma\sb{\mathrm{p}}(A)$
if
there is $\psi\in\dom(A)\setminus\{0\}$ such that $(A-\lambda I_\bfX)\psi=0$;
we say that
$\lambda$
is from the discrete spectrum
$\sigma\sb{\mathrm{d}}(A)$
if
it is an isolated point in $\sigma(A)$
and $A-\lambda I_\bfX$ is a Fredholm operator,
or, equivalently,
if the corresponding Riesz projection is of finite rank
\cite[III.5]{opus}.
We define the essential spectrum by
\begin{eqnarray}\label{def-ess}
\sigma\sb{\mathrm{ess}}(A)
=\sigma(A)\setminus\sigma\sb{\mathrm{d}}(A).
\end{eqnarray}
The definition \eqref{def-ess}
of the essential spectrum
coincides with the \emph{Browder spectrum}
$\sigma\sb{\mathrm{ess},5}(A)$
from \cite[\S I.4]{edmunds2018spectral}
(see \cite[Appendix B]{hundertmark2007exponential}
and \cite[Theorem III.125]{opus}).

\begin{definition}[Virtual levels]
\label{def-virtual}
Let $A\in\scrC(\bfX)$
and let $\varOmega\subset\C\setminus\sigma(A)$
be
such that $\sigma\sb{\mathrm{ess}}(A)\cap\p\varOmega\ne\emptyset$.
Let $\bfE$, $\bfF$ be Banach spaces
with continuous (not necessarily dense) embeddings
\[
\bfE\mathop{\longhookrightarrow}\limits\sp{\imath}
\bfX\mathop{\longhookrightarrow}\limits\sp{\jmath}\bfF.
\]
We say that
$z_0\in\sigma\sb{\mathrm{ess}}(A)\cap\p\varOmega$
is a
\emph{point of the essential spectrum of rank $r\in\N_0$ relative to
$(\varOmega,\bfE,\bfF)$}
if it is the smallest nonnegative integer
for which there is
an operator
$\calB\in\scrB_{00}(\bfF,\bfE)$
of rank $r$
such that
$\varOmega\cap\sigma(A+B)\cap\mathbb{D}_\delta(z_0)=\emptyset$
with some $\delta>0$,
where $B=\imath\circ\calB\circ\jmath\in\scrB_{00}(\bfX)$,
and such that there exists the following limit
in the weak
operator topology\footnote{or in the weak$^*$ operator topology
in the case when $\bfF$ has a pre-dual;
for details, see \cite{virtual-levels}.}
of $\scrB(\bfE,\bfF)$:
\begin{eqnarray}\label{lim}
(A+B-z_0 I_\bfX)^{-1}_{\varOmega,\bfE,\bfF}:=
\wlim\sb{z\to z_0,\,z\in\varOmega\cap\mathbb{D}_\delta(z_0)}
\jmath\circ(A+B-z I_\bfX)^{-1}\circ\imath
:\;\bfE\to\bfF.
\end{eqnarray}
\end{definition}

\begin{itemize}
\item
If $r=0$,
so that there is a limit \eqref{lim}
with $B=0$,
then $z_0$ is called
a \emph{regular point of the essential spectrum
relative to $(\varOmega,\bfE,\bfF)$};
then we say that the resolvent of $A$
\emph{satisfies the limiting absorption principle
at $z_0$
relative to $(\varOmega,\bfE,\bfF)$}.

\item
If $r\ge 1$,
then $z_0$ is called
an \emph{exceptional point of rank $r$}
relative to $(\varOmega,\bfE,\bfF)$.
We will also say that
$z_0$ is a virtual level
of rank $r$
relative to $(\varOmega,\bfE,\bfF)$.

\item
If $\Psi\in\bfF$
is in $\range\big((A+B-z_0 I_\bfX)^{-1}_{\varOmega,\bfE,\bfF}\big)$
(with $B=\imath\circ\calB\circ\jmath$, $\calB\in\scrB_{00}(\bfF,\bfE)$,
is such that the limit \eqref{lim} exists)
and satisfies $(\hatA-z_0 I_\bfF)\Psi=0$ and $\Psi\neq 0$,
then $\Psi$ is called a \emph{virtual state} of $A$
relative to $(\varOmega,\bfE,\bfF)$
corresponding to $z_0$.
(By \cite{virtual-levels},
$\range\big((A+B-z_0 I_\bfX)^{-1}_{\varOmega,\bfE,\bfF}\big)$
does not depend on
$B=\imath\circ\calB\circ\jmath$,
$\calB\in\scrB_{00}(\bfF,\bfE)$,
as long as the limit
\eqref{lim} exists.)
\end{itemize}

We assume that $A\in\scrC(\bfX)$ has a closable extension onto $\bfF$,
in the following sense:
\begin{assumption}
\label{ass-virtual}
The operator
$A\in\scrC(\bfX)$,
considered as a mapping
$\bfF\to\bfF$,
\begin{eqnarray}\label{def-a-f-f}
\dom(A\sb{\bfF\shortto\bfF}):=\jmath(\dom(A)),
\qquad
A\sb{\bfF\shortto\bfF}:\,
\Psi\mapsto\jmath(A\,\jmath^{-1}(\Psi)),
\end{eqnarray}
is closable in $\bfF$, with closure $\hatA\in\scrC(\bfF)$
and domain
$\doma\supset
\dom(A\sb{\bfF\shortto\bfF}):=\jmath\big(\dom(A)\big)$.
\end{assumption}

In the applications of the theory of virtual levels
and virtual states
to differential operators
it is useful to be able to consider
relatively compact perturbations,
allowing in place of
$\calB\in\scrB_{00}(\bfF,\bfE)$ in
Definition~\ref{def-virtual}
operators
$\calB:\,\bfF\to\bfE$
which are $\hatA$-compact.
We note that, by \cite{virtual-levels},
if $A\in\scrC(\bfX)$ satisfies Assumption~\ref{ass-virtual}
and $\calB:\,\bfF\to\bfE$ is $\hatA$-compact,
then $B=\imath\circ\calB\circ\jmath:\,\bfX\to\bfX$
is $A$-compact.

\begin{remark}
The existence of a closed extension
of $\Delta$ from $L^2$ to $L^2_{s}$,
$s\in\R$,
is proved in \cite[Appendix B]{virtual-levels}.
\end{remark}

\section{Jost solutions
for higher order differential operators}

\label{sect-jost-N}

The construction of Jost solutions
for higher order differential operators
closely follows the approach for the Schr\"odinger
operators given by
Faddeev \cite[Lemmata 1.1 -- 1.3]{faddeev1963inverse},
who attributes the approach
to Jost, Bargmann, and Levinson \cite{jost1947falschen,bargmann1949connection,levinson1949uniqueness}.
See \cite[Appendix]{faddeev1963inverse} for the story of the subject.
There are expositions by many authors,
see e.g.
\cite{MR897106,deift1979inverse,bolle1985complete,chadan1989inverse};
here, in particular, we closely follow
the treatment provided in \cite[pp.~325--326]{chadan1989inverse}.

We will use the following notations:
\[
x^\pm=\abs{x}\unity_{\R_\pm}(x),
\quad x\in\R,
\quad
\mbox{so that}
\quad
\langle x^{-}\rangle
=
\begin{cases}
\langle x\rangle,&x<0,
\\
1,&x\ge 0,
\end{cases}
\qquad
\langle x^{+}\rangle
=
\begin{cases}
1,&x\le 0,
\\
\langle x\rangle,&x>0.
\end{cases}
\]
Let $N\in\N$,
$N\ge 2$,
and 
let $V$ be a measurable complex-valued function on $\R$
and assume that there is $\mu>0$ and $C>0$ such that $V$ satisfies
\begin{align}\label{v-small}
\abs{V(x)}\le C e^{-3\mu\abs{x}},
\qquad
x\in\R.
\end{align}
Then
\begin{eqnarray}\label{first-moment}
M:=
\int\sb{\R} \langle x\rangle^{N-1}
e^{\mu\abs{x}}
\abs{V(x)}\,dx<\infty
\end{eqnarray}
and the functions
\[
M_{+}(x)=\int_x^{+\infty}e^{\mu\abs{x}}\langle y\rangle^{N-1}\abs{V(y)}\,dy,
\qquad
M_{-}(x)=\int_{-\infty}^x e^{\mu\abs{x}}\langle y\rangle^{N-1}\abs{V(y)}\,dy,
\qquad
x\in\R
\]
satisfy
\begin{align}\label{first-moment-2}
M_{+}(x)\le C_{+} e^{-\mu\abs{x}},
\quad x\ge 0;
\qquad
M_{-}(x)\le C_{-} e^{-\mu\abs{x}},
\quad x\le 0
\end{align}
with some $C_\pm>0$.

We consider the spectral problem
for the higher order Schr\"odinger operator
$A=(-\jj\p_x)^N+V$
in $L^2(\R)$
with domain $\dom(A)=H^N(\R)$:
\begin{eqnarray}\label{stationary-d1}
\big((-\jj\p_x)^N+V(x)\big)\psi = z\psi,
\qquad
\psi(x)\in\C,
\quad
x\in\R,
\quad
z\in\C.
\end{eqnarray}

Denote $\alpha=e^{\frac{2\pi}{N}\jj}$.

\begin{theorem}\label{theorem-jost}
For each
$\zeta\in\overline{\Gamma_N}$
where $\Gamma_N=
\{z\in\C\sothat
0\le \arg(\zeta)\le\pi/N\}$
and $m\in\N_0$, $m\le N-1$,
equation \eqref{stationary-d1} has
a distributional solution
$\theta_{m}(x,\zeta)$
which is continuous
for all $x\in\R$
and $\zeta\in\overline{\Gamma_N\cap\mathbb{D}_\mu}$,
continuously differentiable in $x\in\R$,
are analytic in $\zeta$
for each $x\in\R$,
and for each
$\zeta\in\overline{\Gamma_N\cap\mathbb{D}_\mu}$ satisfies the asymptotics
\begin{align}\label{as-0}
\lim\sb{x\to+\infty}
\big(\theta(x,\zeta)
-e^{\jj\alpha^m\zeta x}
\big)\to 0,
\qquad
\zeta\in\overline{\Gamma_N\cap\mathbb{D}_\mu}.
\end{align}
This solution satisfies the following estimates
for all $x\in\R$ and for all
$\zeta\in\overline{\Gamma_N\cap\mathbb{D}_\mu}$:
\begin{eqnarray}
\label{thetaestimate}
\abs{\theta_{m}(x,\zeta)}
&\le&
\langle x^{-}\rangle^{N-1}
e^{\frac{3M_{+}(x)}{2\langle\zeta\rangle^{N-1}}}
e^{\abs{\zeta}\abs{x}},
\\[1ex]
\label{thetaestimate-1}
\abs{\theta_{m}(x,\zeta)}
&\le&
e^{\fra{M_{+}(x)}{\abs{\zeta}^{N-1}}}
e^{\abs{\zeta}\abs{x}},
\quad
\zeta\ne 0,
\\[0.5ex]
\abs{\theta_{m}(x,\zeta)-e^{\jj\alpha^m\zeta x}}
&\le&
\frac{3\langle x^{-}\rangle^{N-1}}{2\langle \zeta \rangle^{N-1}}
e^{\frac{3M_{+}(x)}{2\langle\zeta\rangle^{N-1}}}
e^{\abs{\zeta}\abs{x}}
M_{+}(x).
\label{theta-1-such-0}
\end{eqnarray}
Further,
for all $x\in\R$ and
$\zeta\in\overline{\Gamma_N\cap\mathbb{D}_\mu}$,
\begin{align}
&
\abs{\p_x^{N-1}
\theta_{m}(x,\zeta)-(\jj\alpha^m\zeta)^{N-1}e^{\jj\alpha^m\zeta x}}
\nonumber
\\
&
\le
e^{\frac{3M_{+}(x)}{2\langle\zeta\rangle^{N-1}}}
\Big(
M_{+}(x)
e^{\abs{\zeta}\abs{x}}
+
\frac{3\abs{\zeta}}{2}
\int\limits_x^\infty
\langle y^{-}\rangle^{N-1}
M_{+}(y)
e^{\abs{\zeta}\abs{y}}
\,dy
\Big)
\nonumber
\\
&
\le
C_0
\begin{cases}
e^{-\mu\abs{x}}(1+\abs{\zeta}),
&x\ge 0,
\\
e^{\abs{\zeta}\abs{x}}(1+\abs{\zeta}\langle x\rangle^{N}),
&x\le 0,
\end{cases}
\label{theta-1-prime}
\end{align}
for some $C_0>0$.
More generally,
for any $k\in\N_0$, $k\le N-2$,
there is $C>0$
such that
for all $x\in\R$,
$\zeta\in\overline{\Gamma_N\cap\mathbb{D}_\mu}$,
\begin{align}
\label{theta-1-prime-2}
&\abs{\p_x^{N-1-k}
\theta_{m}(x,\zeta)-(\jj\alpha^m\zeta)^{N-1-k}e^{\jj\alpha^m\zeta x}}
\le
C
\begin{cases}
e^{-\mu\abs{x}}
(1+\abs{\zeta}),
&x\ge 0,
\\
e^{\abs{\zeta}\abs{x}}
(\langle x\rangle^k+\abs{\zeta}\langle x\rangle^{N+k}),
&x\le 0.
\end{cases}
\end{align}
\end{theorem}

\section{Constructing the resolvent from the
Jost solutions}

\label{sect-resolvent}

We consider the equation
\begin{align}\label{dvz}
\big((-\jj\p_x)^N+V(x)\big)\psi=z\psi.
\end{align}
Let
$\zeta\in\Gamma_N$,
$\zeta^N=z\in\C_{+}$.
The fundamental solution to
$\big((-\jj\p_x)^N-\zeta^N\big)\psi=0$ is given by
\begin{align}\label{def-g0}
G_0(x,\zeta)
=
\theta(x)
\begin{cases}
\frac{\jj}{N}
\sum_{j=0}^{N-1}\frac{e^{\jj\alpha^j\zeta x}}{(\alpha^j\zeta)^{N-1}},
&\zeta\ne 0;
\\[1ex]
\jj^N
\frac{x^{N-1}}{(N-1)!},
&\zeta=0.
\end{cases}
\end{align}
We note that for any $k\in\N_0$, $k\le N$,
one has:
\begin{align}\label{g0-prime-zero}
(-\jj\p_x)^k G_0(x,\zeta)\at{x=0+}
=
\begin{cases}
0,&0\le k\le N-2,
\\
1,&k=N-1;
\end{cases}
\qquad
\zeta\in\overline{\Gamma_N}.
\end{align}
This is immediate from \eqref{def-g0} for $\zeta=0$;
for $\zeta\ne 0$,
this follows from the identity
\begin{eqnarray}\label{the-id}
\frac{1}{N}
\sum_{j=0}^{N-1}\frac{\alpha^{jr}}{\alpha^{j(N-1)}}
=
\frac{1}{N}
\sum_{j=0}^{N-1}\alpha^{j(r+1-N)}=
\begin{cases}
0,&r\in\N_0,\ r\le N-2;
\\
1,&r=N-1;
\end{cases}
\end{eqnarray}
the above relation follows
after we notice that
$\alpha^{k+1-N}$,
with $0\le k\le N-2$,
is a root of $1$ of order $N$
which is different from $1$.

For $\zeta\in\overline{\Gamma_N}$,
there are
$[(N+1)/2]$
Jost solutions
$\theta_m(x,\zeta)$
to
\[
\big((-\jj\p_x)^N+V\big)u=z u,
\qquad
z=\zeta^N\in\overline{\C_{+}},
\]
with asymptotics
$\theta_m(x,\zeta)\sim e^{\jj\alpha^m\zeta x}$
for $x\to+\infty$,
decaying (or bounded) for large positive $x$,
and
$[N/2]$
Jost solutions $\gamma_m(x,\zeta)$
decaying (or bounded) like
$\theta_m(x,\zeta)\sim e^{\jj\alpha^m\zeta x}$
for $x\to-\infty$,
satisfying bounds similar to
the ones in Theorem~\ref{theorem-jost}.

The value $z\in\C\setminus\sigma\sb{\mathrm{ess}}(A)$
is an eigenvalue
if Jost solutions
$\theta_j(x,\zeta)$, $1\le j\le [(N+1)/2]$,
and
$\gamma_j(x,\zeta)$
$1\le j\le [N/2]$,
are linearly dependent at
$\zeta\in\Gamma_N$ satisfying $\zeta^N=z$.
This happens if
\begin{align}
\det
\begin{bmatrix}
\theta_1(x,\zeta)&\theta_2(x,\zeta)&\dots&
\gamma_{N-1}(x,\zeta)&\gamma_{N}(x,\zeta)
\\
\p_x\theta_1(x,\zeta)&\p_x\theta_2(x,\zeta)&\dots&
\p_x\gamma_{N-1}(x,\zeta)&\p_x\gamma_{N}(x,\zeta)
\\
\dots
\\
\p_x^{N-1}\theta_1(x,\zeta)&\p_x^{N-1}\theta_2(x,\zeta)&\dots&
\p_x^{N-1}\gamma_{N-1}(x,\zeta)&\p_x^{N-1}\gamma_{N}(x,\zeta)
\end{bmatrix}
=0.
\end{align}

If $z=\zeta^N$ is not an eigenvalue,
then
there are $c_j(y,\zeta)$
and $k_j(y,\zeta)$ such that
\begin{align}
G(x,y;\zeta)
=\begin{cases}
\sum\sb{j}k_j(y,\zeta)\gamma_j(x,\zeta),&x\le y
\\
\sum\sb{j}c_j(y,\zeta)\theta_j(x,\zeta),&x\ge y
\end{cases}
\end{align}
is a fundamental solution to
$A-z=(-\jj\p_x)^N+V-z$,
with $\p_x^j G(x,y;\zeta)$
satisfying the continuity condition
at $x=y$ as long as $0\le j\le N-2$
and satisfying the jump condition
\begin{align}\label{jump}
(-\jj\p_x)^{N-1} G(x,y;\zeta)\at{x=y+0}
-
(-\jj\p_x)^{N-1} G(x,y;\zeta)\at{x=y-0}
=\jj.
\end{align}
Let us give the explicit construction in the case
$N=3$.
We assume that there are two Jost solutions
$\theta_1$, $\theta_2$
and one solution $\gamma_1$.
We have the following system
at $x=y$:
\[
\begin{cases}
c_1(y,\zeta)\theta_1(y,\zeta)
+c_2(y,\zeta)\theta_2(y,\zeta)
-k_1(y,\zeta)\gamma_1(y,\zeta)
=0,
\\
c_1(y,\zeta)\theta_1'(y,\zeta)
+c_2(y,\zeta)\theta_2'(y,\zeta)
-k_1(y,\zeta)\gamma_1'(y,\zeta)'
=0,
\\
c_1(y,\zeta)\theta_1''(y,\zeta)
c_2(y,\zeta)\theta_2''(y,\zeta)
-k_1(y,\zeta)\gamma_1''(y,\zeta)
=\jj^{N-1},
\end{cases}
\]
\ac{The powers of $\jj$ to be checked!!}
where primes denote derivatives with respect
to the first variable.
This can be written as
\[
\begin{bmatrix}
\theta_1&\theta_2&\gamma_1
\\
\theta_1'&\theta_2'&\gamma_1'
\\
\theta_1''&\theta_2''&\gamma_1''
\end{bmatrix}
\begin{bmatrix}c_1\\c_2\\-k_1\end{bmatrix}
=
\begin{bmatrix}0\\0\\\jj^{N-1}\end{bmatrix},
\]
hence
\begin{align}\label{cck}
\begin{bmatrix}c_1\\c_2\\-k_1\end{bmatrix}
=
\frac{1}{\varDelta(\zeta)}
\begin{bmatrix}
\cdot&\cdot&\theta_2\gamma_1'-\theta_2'\gamma_1
\\
\cdot&\cdot&\theta_1'\gamma_1-\theta_1\gamma_1'
\\
\cdot&\cdot&\theta_1\theta_2'-\theta_1'\theta_2
\end{bmatrix}
\begin{bmatrix}0\\0\\\jj^{N-1}\end{bmatrix}
=
\frac{\jj^{N-1}}{\varDelta(\zeta)}
\begin{bmatrix}
\theta_2\gamma_1'-\theta_2'\gamma_1
\\
\theta_1'\gamma_1-\theta_1\gamma_1'
\\
\theta_1\theta_2'-\theta_1'\theta_2
\end{bmatrix}
=
\frac{\jj^{N-1}}{\varDelta(\zeta)}
\begin{bmatrix}
\{\theta_2,\gamma_1\}
\\
\{\gamma_1,\theta_1\}
\\
\{\theta_1,\theta_2\}
\end{bmatrix},
\end{align}
with
\begin{align}\label{def-var-delta}
\varDelta(\zeta)
=
\det
\begin{bmatrix}
\theta_1&\theta_2&\gamma_1
\\
\theta_1'&\theta_2'&\gamma_1'
\\
\theta_1'&\theta_2''&\gamma_1''
\end{bmatrix},
\qquad
\{\theta_1,\theta_2\}=
\theta_1\p_x\theta_2
-\theta_2\p_x\theta_1.
\end{align}
\begin{remark}
We note that
$\varDelta(\zeta)$
defined in \eqref{def-var-delta}
indeed
depends only on $\zeta$ but not on $x$.
Indeed,
$\begin{bmatrix}\theta_i\\\theta_i'\\\theta_i''\end{bmatrix}$,
with $i=1,\,2$, and 
$\begin{bmatrix}\gamma_1\\\gamma_1'\\\gamma_1''\end{bmatrix}$
satisfy the equation
$
\p_x w(x,\zeta)
=
\begin{bmatrix}0&1&0\\0&0&1\\\jj^{N-1}(z-V)&0&0\end{bmatrix}
w(x,\zeta)
$
(cf. \eqref{dvz}),
and, by Liouville's formula,
$\varDelta(x,\zeta)$
satisfies
\[
\p_x
\varDelta(x,\zeta)
=
\tr
\begin{bmatrix}0&1&0\\0&0&1\\\jj^{N-1}(z-V)&0&0\end{bmatrix}
\varDelta(x,\zeta)
=0.
\]
\end{remark}

The relation \eqref{cck}
leads to
\begin{align}
G(x,y;\zeta)
&=
\begin{cases}
k_1(y,\zeta)\gamma_1(x,\zeta)
&x\le y,
\\
c_1(y,\zeta)\theta_1(x,\zeta)+c_2(y,\zeta)\theta_2(x,\zeta)
&x\ge y,
\end{cases}
\nonumber
\\
&=
\frac{\jj^{N-1}}{\varDelta(\zeta)}
\begin{cases}
\{\theta_2,\theta_1\}(y)
\gamma_1(x)
&x\le y,
\\
\{\theta_2,\gamma_1\}(y)
\theta_1(x)
+
\{\gamma_1,\theta_1\}(y)
\theta_2(x)
&x\ge y.
\end{cases}
\label{g-is}
\end{align}
In the second line,
we did not indicate explicitly the dependence
of the Jost solutions on $\zeta$.

We note that as $\zeta\to 0$,
both $\theta_1(x,\zeta)$ and $\theta_2(x,\zeta)$
pointwise converge to the same function,
$\theta_1(x,0)=\theta_2(x,0)\to 1$  as $x\to+\infty$,
hence $\varDelta(\zeta)$ vanishes at $\zeta=0$.
At the same time,
this means that
in the expression \eqref{g-is}
$\{\theta_2,\theta_1\}(y,\zeta)$ also goes to zero
(pointwise in $y$) as $\zeta\to 0$,
and
also
$
\{\theta_2,\gamma_1\}(y,\zeta)
\theta_1(x,\zeta)
+
\{\gamma_1,\theta_1\}(y,\zeta)
\theta_2(x,\zeta)
$
goes to zero (pointwise in $x$ and $y$)
as $\zeta\to 0$.
As a result,
$G(x,y;\zeta)$
remains bounded pointwise in $x$, $y$.

This suggests that the resolvent has a limit
as $\zeta\to 0$
if $\varDelta(\zeta)$ has a zero at $\zeta=0$
of order one;
this is what we are going to show
in the rest of this section.

\begin{lemma}\label{lemma-41}
Let $V\in C^\infty_{[-1,1]}(\R)$.

\begin{enumerate}
\item
Assume that
for $x>1$,
as $\zeta\to 0$,
$\gamma_1(x,\zeta)$
converges to $a+bx+cx^2$.
Then $\varDelta(\zeta)$ vanishes simply
at $\zeta=0$
if and only if $c=0$.
\item
Assume that
for $x<1$,
as $\zeta\to 0$,
$\theta_j(x,\zeta)$, $j=1,\,2$
converge to $a+bx+cx^2$.
Then $\varDelta(\zeta)$ vanishes of
order at least two
at $\zeta=0$
if $b=c=0$.
\end{enumerate}
\end{lemma}

\begin{proof}
This is a direct computation.
For $x>1$
one has
$\theta_1(x,\zeta)=e^{\jj\zeta x}$,
$\theta_2(x,\zeta)=e^{\jj\alpha\zeta x}$,
$\gamma_1=a+bx+cx^2+O(\zeta)$,
hence
\begin{align*}
&
\det
\begin{bmatrix}
\theta_1&\theta_2&\gamma_1
\\
\theta_1'&\theta_2'&\gamma_1'
\\
\theta_1''&\theta_2''&\gamma_1''
\end{bmatrix}
=
\det
\begin{bmatrix}
\theta_1&\theta_2&a+bx+cx^2+O(\zeta)
\\
\jj\zeta\theta_1&\jj\alpha\zeta\theta_2&b+2cx+O(\zeta)
\\
-\zeta^2\theta_1&-\alpha^2\zeta^2\theta_2&2c+O(\zeta)
\end{bmatrix}
\\
&=
\theta_1\theta_2
\det
\begin{bmatrix}
1&1&a+bx+cx^2+O(\zeta)
\\
\jj\zeta&\jj\alpha\zeta&b+2cx+O(\zeta)
\\
0&0&2c+O(\zeta)
\end{bmatrix}
+O(\zeta^2)
=
2c\theta_1\theta_2
\jj(\alpha-1)\zeta
+O(\zeta^2).
\end{align*}

For $x<-1$
one has
$\gamma(x,\zeta)=e^{\jj\alpha^2\zeta x}$,
$\theta_1(x,\zeta),
\,\theta_2=A+Bx+Cx^2+O(\zeta)$
(with the same $A,\,B,\,C$ but different
remainder $O(\zeta)$),
hence
\begin{align*}
&
\det
\begin{bmatrix}
\theta_1&\theta_2&\gamma_1
\\
\theta_1'&\theta_2'&\gamma_1'
\\
\theta_1''&\theta_2''&\gamma_1''
\end{bmatrix}
=
\det
\begin{bmatrix}
\gamma_1
&A+Bx+Cx^2+O(\zeta)&A+Bx+Cx^2+O(\zeta)
\\
\jj\alpha^2\zeta\gamma_1
&B+2Cx+O(\zeta)&B+2Cx+O(\zeta)
\\
0&2C+O(\zeta)&2C+O(\zeta)
\end{bmatrix}
+O(\zeta^2)
\\
&=
\gamma_1
\det
\begin{bmatrix}
B+2Cx+O(\zeta)&B+2Cx+O(\zeta)
\\
2C+O(\zeta)&2C+O(\zeta)
\end{bmatrix}
+O(\zeta^2).
\end{align*}
\end{proof}

\begin{corollary}
\label{corollary-no-lap}
Let $N=3$.
If $V\in C_{\mathrm{comp}}(\R,\C)$
and if there is a solution
$\Psi(x)$ to $\big((-\jj\p_x)^N+V\big)u=0$
such that
$\Psi(x)\to 1$ as $x\to\pm\infty$,
$\langle x\rangle^{2-N}\Psi\in L^\infty(\R,\C)$,
then $\zeta=0$ is a virtual level
relative to $L^2_\nu,\,L^2_{-\nu},\,\C_{+}$,
for arbitrarily large $\nu>0$,
in the sense that
$\big((-\jj\p_x)^N+V-z I\big)^{-1}$
does not have a limit as $z\to 0$,
$\Im z>0$,
in the weak topology of
$\scrB\big(L^2_{\nu}(\R,\C),L^2_{-\nu}(\R,\C)\big)$.
\end{corollary}

Above,
for $\nu\in\R$,
\[
L^2_\nu(\R,\C)
=\big\{
f\in L^2_{\mathrm{loc}}(\R,\C)
\sothat
e^{\nu\abs{x}}f\in L^2(\R,\C)
\big\},
\qquad
\norm{f}_{L^2_\nu}
=
\norm{e^{\nu\abs{x}}f}_{L^2}.
\]

\begin{proof}
It suffices to notice that
in \eqref{g-is}
the coefficient
$\{\theta_1,\theta_2\}(y,\zeta)$
vanishes simply in $\zeta$ as $\zeta\to 0$,
while the denominator as at $\zeta=0$
the zero of order at least two
(cf. Lemma~\ref{lemma-41})
if and only if
there is the solution $\Psi$ as specified in the lemma.
\end{proof}

We take into account
the estimates from Theorem~\ref{theorem-jost},
\[
\abs{
 \p_x^{N-1-k}\theta_m(x,\zeta)-(\jj\alpha^m\zeta)^{N-1-k}
e^{\jj\alpha^m \zeta x}
}
\le
C
\begin{cases}
e^{-\mu\abs{x}}
(1+\abs{\zeta}),
&x\ge 0,
\\
e^{\abs{\zeta}\abs{x}}
(\langle x\rangle^k+\abs{\zeta}\langle x\rangle^{N+k}),
&x\le 0,
\end{cases}
\]
concluding that there is $C>0$ such that
for $x\in\R$ and
$\zeta\in\overline{\Gamma_N\cap\mathbb{D}_\mu}$
one has:
\[
\abs{\p_x^{N-1-k}\theta_m(x,\zeta)}
\le
C
\begin{cases}
e^{-x\Im(\alpha^m\zeta)}(1+\abs{\zeta}),
&x\ge 0,
\\
e^{\abs{\zeta}\abs{x}}
(\langle x\rangle^k+\abs{\zeta}\langle x\rangle^{N+k}),
&x\le 0.
\end{cases}
\]
In particular,
\[
\abs{\p_x\theta_m(x,\zeta)}
\le
C
\begin{cases}
e^{-x\Im(\alpha^m\zeta)}
(1+\abs{\zeta}),
&x\ge 0,
\\
e^{\abs{\zeta}\abs{x}}
(\langle x\rangle^{N-1}+\abs{\zeta}\langle x\rangle^{2N-1}),
&x\le 0.
\end{cases}
\]
Also, by \eqref{thetaestimate-1},
\[
\abs{\theta_m(x,\zeta)}
\le
C
\begin{cases}
e^{-x\Im(\alpha^m\zeta)},
&x\ge 0,
\\
e^{\abs{\zeta}\abs{x}}
\langle x\rangle^{N-1},
&x\le 0.
\end{cases}
\]
We also have
\begin{align}\label{hospital}
\abs{\p_\zeta\theta_m(x,\zeta)}
\le
C
\langle x\rangle
\begin{cases}
e^{-x\Im(\alpha^m\zeta)},
&x\ge 0,
\\
e^{\abs{\zeta}\abs{x}}
\langle x\rangle^{N-1},
&x\le 0;
\end{cases}
\end{align}
the estimate \eqref{hospital}
is obtained in the same way
as other estimates on Jost solutions;
we note that
the estimate holds
with some $C>0$
for $x,\,y,\,\zeta$ from a compact set,
while for $x$ large,
$\theta_m$ is a linear combination
of $e^{\jj\alpha_k\zeta x}$,
so a derivative in $\zeta$ produces a factor of $x$.

The above estimates
allow us to write the following (non-optimal) bound
on $G(x,y;\zeta)$:
for some $C>0$ and for
$\zeta\in\overline{\Gamma_N\cap\mathbb{D}_\mu}$,
\begin{align}\label{non-optimal}
\abs{G(x,y;\zeta)}
\le
C
\begin{cases}
e^{2\abs{\zeta}\abs{x}}\langle x\rangle^{3N-2}
,&0\le x\le y,
\\
1
,&x\le 0\le y,
\\
e^{2\abs{\zeta}\abs{y}}
\langle y\rangle^{3N-2}
,&x\le y\le 0;
\\
e^{2\abs{\zeta}\abs{y}}
\langle y\rangle^{3N-2}
,&0\le y\le x,
\\
1
,&y\le 0\le  x,
\\
e^{2\abs{\zeta}\abs{x}}\langle x\rangle^{3N-2}
,& y\le x\le 0.
\end{cases}
\end{align}
We note that
the factor
$\langle x\rangle+\langle y\rangle$
comes from
applying l'H\^{o}pital's rule
to \eqref{g-is}, in the form
\[
\Abs{
\frac{\{\theta_2,\theta_1\}(y)\gamma_1(x)}
{\varDelta(\zeta)}
}
\le
\Abs{
\frac{\p_\zeta\{\theta_2,\theta_1\}(y)\gamma_1(x)}
{\p_\zeta\varDelta\zeta}
}
\le
C\Abs{
\p_\zeta\{\theta_2,\theta_1\}(y)\gamma_1(x)},
\]
and using
\eqref{hospital}.
The estimates
\eqref{non-optimal}
are not optimal,
but they are uniform for
$\zeta\in\overline{\Gamma_N\cap\mathbb{D}_\mu}$.
They prove the following result:

\begin{lemma}
\label{lemma-lap-nu}
Assume that $\varDelta(\zeta)$ vanishes simply at $\zeta=0$.
Then the resolvent $(A-z I)^{-1}$, $z\in\C_{+}$,
is uniformly bounded for
$z\in\C_{+}\cap\mathbb{D}_\mu$ and converges,
as $z\to z_0$,
in the uniform operator topology
of
$L^{2}_\nu(\R,\C)\to L^{2}_{-\nu}(\R,\C)$,
with $\nu>3\mu$.
\end{lemma}

In the limit $\zeta\to 0$,
$0<\arg(\zeta)<\pi/3$,
the integral kernel
$G(x,y;\zeta)$ converges pointwise to
$G_0(x,y)$ which satisfies the following bounds:
\begin{align}\label{G-estimates-weak-2}
\abs{G_0(x,y)}
\le
C
\min(\langle x\rangle^2,\langle y\rangle^2),
\end{align}
with some $C>0$.
We notice that $\theta(x,\zeta)\sim e^{\jj\alpha^m\zeta x}$
for $x\to+\infty$
which converges to $\theta_0(x)=1$ for $x\gg 1$
will have the asymptotics 
$\theta_0(x)=a+bx+cx^2+o(1)$ for $x\to-\infty$,
with some $a,\,b,\,c\in\C$;
Similarly, $\gamma(x,\zeta)\sim e^{\jj\alpha^m\zeta x}$
for $x\to-\infty$
which converges to $\gamma_0(x)=1$ for $x\ll -1$
will have the asymptotics
$\gamma_0(x)=a'+b'x+c'x^2+o(1)$ for $x\to+\infty$,
with some $a',\,b',\,c'\in\C$;
Similarly, a function
$\Theta_0(x)$
which has an asymptotic
$\Theta(x)=x+o(1)$ for $x\to+\infty$
will have the asymptotics 
$\Theta_0(x)=A+Bx+Cx^2+o(1)$ for $x\to-\infty$,
with some $A,\,B,\,C\in\C$.
One can now construct $G_0$ out of $\theta_0$, $\gamma_0$,
and $\Theta_0$.
We also notice that $\Theta_0\sim x$ for $x\gg 1$
can be obtained as a pointwise limit
of $\theta_1(x,\zeta)$ and $\theta_2(x,\zeta)$
by taking particular coefficients;
in particular,
as $\zeta\to 0$,
$\frac{e^{\jj\zeta x}-e^{\jj\alpha\zeta x}}{\jj(1-\alpha)\zeta}$
converges pointwise to $x$.

This allows to conclude with the following lemma:

\begin{lemma}\label{lemma-n}
Assume that
$\varDelta(\zeta)$
vanishes simply at $\zeta=0$.
Then
$R_0=\wlim\sb{z\to 0}(A-z I)^{-1}:\;
L^2_s(\R,\C)\to L^2_{-s}(\R,\C)$,
$s>3\mu$
extends to a continuous mapping
\[
L^2_{s}(\R,\C)
\to
L^2_{-s'}(\R,\C),
\qquad
s,\,s'>N-3/2,
\qquad
s+s\ge N.
\]
\end{lemma}
(The second inequality is redundant if $N\ge 3$.)

The above lemma follows from the estimate
\eqref{G-estimates-weak-2}
and the following result:

\begin{lemma}
\label{lemma-estimates}
Assume that
the integral kernel of the operator
$G:\,\scrD(\R,\C)\to\scrD'(\R,\C)$
satisfies the estimate
\[
\abs{\calK(G)(x,y)}\le C
\max(\langle x\rangle,\langle y\rangle)^{N-1},
\qquad
x,\,y\in\R.
\]
Then $G$
extends to a continuous mapping
\[
L^2_{s}(\R,\C)
\to
L^2_{-s'}(\R,\C),
\qquad
s,\,s'>N-3/2,
\qquad
s+s\ge N.
\]
\end{lemma}

\begin{proof}
To prove the $L^2_s\to L^2_{-s'}$ estimates
in the case $s,\,s'>1/2$, $s+s'=N$,
one can decompose $G=\sum\sb{\sigma,\,\sigma'\in\{\pm\}}
\unity_{\R_\sigma}\circ G\circ\unity_{\R_{\sigma'}}$;
it suffices to consider $\unity_{\R_{+}}\circ G\circ\unity_{\R_{+}}$.
It is enough to show that the operators
$G_1,\,G_2:\,\mathscr{D}(\R_{+})\to\mathscr{D}'(\R_{+})$
with the integral kernels
\begin{eqnarray}\label{def-t-kernel-2}
\calK_1(x,y)
=
\unity_{\R_{+}}(x)
\langle y\rangle^{N-1}
\unity_{[1,x]}(y),
\quad
\calK_2(x,y)
=
\unity_{[1,y]}(x)
\langle x\rangle^{N-1}
\unity_{\R_{+}}(y),
\quad
x,\,y\in\R_{+},
\end{eqnarray}
have the regularity properties announced in the lemma.
This is done by proving the almost orthogonality
\cite{cotlar1955,stein1993harmonic}
of the pieces of the following dyadic partition
of $\calK_1$ (and similarly for $\calK_2$):
\[
\calK_1=\sum\sb{j,\,k\in\N}
\calK_{jk},
\qquad
\calK_{jk}
=
\unity_{[1,2]}(x/2^j)
\unity_{\R_{+}}(x)
\langle y\rangle^{N-1}
\unity_{[1,x]}(y)
\unity_{[1,2]}(y/2^k),
\qquad
j,\,k\in\N_0.
\qedhere
\]
\end{proof}

We also mention the following relation
convenient for analyzing
\eqref{g-is}:

\begin{lemma}
If $\theta_1$ and $\theta_2$ satisfy
\[
(\jj\p_x^3+V-z)\theta=0,
\]
then
$u(x,\zeta)=\{\theta_1(x,\zeta),\theta_2(x,\zeta)\}$
satisfies
\[
(-\jj\p_x^3+V-z)u=0.
\]
\end{lemma}
\begin{proof}
Indeed, we derive:
\[
\p_x(\theta_1''\theta_2'-\theta_1'\theta_2'')
=
\theta_1'''\theta_2'-\theta_1'\theta_2'''
=\jj(V-z)(\theta_1\theta_2'-\theta_1'\theta_2)
=\jj(V-z)u,
\]
while
$\p_x u=\theta_1\theta_2''-\theta_1''\theta_2$,
$\p_x^2 u
=\theta_1'\theta_2''+\theta_1\theta_2'''-\theta_1''\theta_2'
-\theta_1'''\theta_2
=\theta_1'\theta_2''-\theta_1''\theta_2'$,
so
$-\p_x^3 u=\jj(V-z)u$.
\end{proof}

\section{Resolvent estimates via finite codimension restriction}

\label{sect-finite-codimension}

Let us prove that indeed
not only the limit of the resolvent
at $\zeta=0$ defines a bounded mapping
as stated in Lemma~\ref{lemma-n},
but also that the convergence
of the resolvent
takes place in the uniform
operator topology
of $\scrB\big(L^2_s(\R,\C),L^2_{-s'}(\R,\C)\big)$, $s,\,s'>N-3/2$.

We note that
$e^{\jj\alpha^j\zeta x}$
with $0\le j\le [j/2]$
decay as $x\to+\infty$,
while if $[j/2]<j\le N-1$
they decay as $x\to-\infty$.
Due to \eqref{def-g0},
for $\zeta\in\Gamma_N$,
the resolvent
$\big((-\jj\p_x)^N-\zeta^N I\big)^{-1}$
is represented by the integral operator
with kernel
\begin{align}\label{def-r-kernel}
R(x,y;\zeta)
=
\begin{cases}
\frac{\jj}{N}
\sum_{j=0}^{[N/2]}\frac{e^{\jj\alpha^j\zeta(x-y)}}{(\alpha^j\zeta)^{N-1}},
&
x\ge y;
\\
-\frac{\jj}{N}
\sum_{j={[N/2]+1}}^{N}\frac{e^{\jj\alpha^j\zeta(x-y)}}{(\alpha^j\zeta)^{N-1}},
&
x\le y.
\end{cases}
\end{align}
Let us mention
that one arrives at the above expression subtracting
$e^{\jj\alpha^j\zeta(x-y)}$ with $j\ge [N/2]+1$
from \eqref{def-g0},
which does not change the continuity and the jump conditions;
note that
the integral kernel $R(x,y;\zeta)$
decays for large $x$ and $y$,
giving an operator which is bounded in $L^2(\R,\C)$
(for a particular $\zeta$)
and thus represents
the integral kernel of $\big((-\jj\p_x)^N-\zeta^N I\big)^{-1}$.
Expanding exponents
in \eqref{def-r-kernel}, we find:
\begin{align}\label{def-r0}
R(x,y;\zeta)
&
=
\frac{\jj}{N}
\sum_{j=0}^{[N/2]}
\Big(
\frac{1}{(\alpha^j\zeta)^{N-1}}
+
\frac{x-y}{(\alpha^j\zeta)^{N-2}}
+
\frac{(x-y)^2}{2!(\alpha^j\zeta)^{N-3}}
+\dots
+
\frac{(x-y)^{N-2}}{(N-2)!\alpha^j\zeta}
\Big)
+
R_1(x,y;\zeta)
\nonumber
\\
&
=
\frac{\jj}{N}
\sum_{j=0}^{[N/2]}
\sum_{\ell=0}^{N-2}
\frac{(x-y)^{\ell}}{\ell!(\alpha^j\zeta)^{N-1-\ell}}
+
R_1(x,y;\zeta)
,
\end{align}
valid for all $x,\,y\in\R$,
with $R_1$ given by the Taylor series remainder,
\begin{align}\label{def-r1}
R_1(x,y;\zeta)
=
\begin{cases}
\frac{\jj}{N}\sum_{j=0}^{[N/2]}\frac{(x-y)^{N-1}e^{s_j\jj\alpha^j\zeta(x-y)}}{(N-1)!},
&x>y,
\\
-\frac{\jj}{N}\sum_{j=[N/2]+1}^{N-1}\frac{(x-y)^{N-1}e^{s_j\jj\alpha^j\zeta(x-y)}}{(N-1)!},
&x<y,
\end{cases}
\end{align}
with $s_j\in[0,1]$, $0\le j\le N-1$,
dependent on $x-y$ and $\zeta$.
From \eqref{def-r1} we deduce:
\begin{align}\label{r1-small}
\abs{R_1(x,y)}
\le
\frac{
[(N+1)/2]
}{N!}\abs{x-y}^{N-1}.
\end{align}
Above, $[N+1]/2$
represents
the maximum of the number of terms
in summations in \eqref{def-r-kernel}
(maximum of $M$ and $N-M$).
By Lemma~\ref{lemma-estimates},
the mapping
\[
R_1(\zeta):
\;
L^2_s(\R,\C)\to L^2_{-s'}(\R,\C)
\]
is bounded uniformly in $\zeta$.
We also notice that
$R_1(x,y,\zeta)$ has a limit
(pointwise in $x$, $y$)
as $\zeta\to 0$,
hence $R_1(\zeta)$
converges in the weak operator topology
of $\scrB\big(L^2_s(\R),L^2_{-s'}(\R)\big)$.
Since this result holds for
arbitrary $s,\,s'>N-3/2$,
the convergence also holds
in the uniform operator topology
of $\scrB\big(L^2_s(\R,\C),L^2_{-s'}(\R,\C)\big)$.

Moreover,
by \eqref{def-r0},
the resolvent $R(x,y;\zeta)$
coincides with the regular part $R_1(x,y;\zeta)$
on the subspace
\[
L^1_{N-2,0,\dots,0}(\R,\C)
=
\Big\{
f\in L^1(\R,\C)
\sothat
\int_\R x^j f(x)\,dx=0,
\quad
0\le j\le N-2
\Big\}
\subset L^1_{N-2}(\R,\C).
\]

\noindent
{\bf The case $N=3$.\ }
The resolvent
\eqref{def-r-kernel}
can be written as
\begin{align}\label{def-r-kernel-3}
R(x,y;\zeta)
&
=-\frac{\jj}{3}
\Big(
\frac{1}{(\alpha^2\zeta)^2}
+
\frac{\jj(x-y)}{\alpha^2\zeta}
\Big)
+R_1(x,y;\zeta)
\nonumber
\\
&=\frac{\alpha}{3}
\Big(
-\frac{\jj\alpha}{\zeta^2}
+\frac{x-y}{\zeta}
\Big)
+R_1(x,y;\zeta),
\qquad
x,\,y\in\R,
\end{align}
with $\abs{R_1(x,y;\zeta)}\le 2\abs{x-y}^2/3$
due to \eqref{r1-small}.
Let $\phi_0=1/2\chi_{[-1,1]}(x)$
and $\phi_1(x)=x\chi_{[-1,1]}(x)$;
then
$\langle x^j,\phi_k\rangle=\delta_{j k}$,
$0\le j,\,k\le 1$.
Denote $\varTheta_j(x,\zeta)=(R(\zeta)\phi_j)(x)$,
$0\le j\le 1$.
Using \eqref{def-r-kernel-3}, we compute:
\begin{align*}
R(\zeta)\phi_0
=
-\frac{\jj\alpha^2}{3\zeta^2}
+
\frac{\alpha x}{3\zeta}
+\varTheta_0(x,\zeta),
\qquad
R(\zeta)\phi_1
=
-\frac{\alpha}{3\zeta}
+\varTheta_1(x,\zeta).
\end{align*}
Applying $A-z$ (with $A=(-\jj\p_x)^3$
and $z=\zeta^3$)
to the above relations
and multiplying them by $\zeta^2$ and
$\zeta$, respectively,
we arrive at
\begin{align}
\label{e1-1}
(A-z)
\Big(
-\frac{\jj\alpha^2}{3}
+
\frac{\alpha\zeta x}{3}
+\zeta^2\varTheta_0(x,\zeta)\Big)
=\zeta^2\phi_0,
\qquad
(A-z)
\Big(-\frac{\alpha}{3}
+\zeta\varTheta_1(x,\zeta)
\Big)
=\zeta\phi_1.
\end{align}
Multiplying the second relation by $\alpha\jj$,
taking the difference, and dividing by $\zeta$,
we have:
\begin{align}
\label{e2-2}
&
(A-z)
\Big(
\frac{\alpha x}{3}
+\zeta\varTheta_0(x,\zeta)
-\jj\alpha\varTheta_1(x,\zeta)
\Big)
=
\zeta\phi_0-\jj\alpha\phi_1.
\end{align}
Defining $B_0=\phi_0\otimes\phi_0$,
we rewrite the second equation from
\eqref{e1-1}
and equation \eqref{e2-2} as
\begin{align}
\label{qn1}
(A+B_0-z)
\Big(-\frac{\alpha}{3}
+\zeta\varTheta_1(x,\zeta)
\Big)
&
=
\zeta\phi_1
-
\frac{\alpha}{3}\phi_0
+\zeta B_0\varTheta_1(\zeta);
\\
\label{qn2}
(A+B_0-z)
\Big(\frac{\alpha x}{3}
+\zeta\varTheta_0(x,\zeta)
-\jj\alpha\varTheta_1(x,\zeta)
\Big)
&
=
\zeta\phi_0-\jj\alpha\phi_1
+B_0\big(\zeta\varTheta_0(\zeta)
-\jj\alpha\varTheta_1(\zeta)\big).
\end{align}
Now we can solve
\begin{align}
(A+B_0-z)u=f,
\qquad
f\in L^2_{-s'}(\R,\C),
\quad
s'>N-3/2,
\end{align}
for $z\in\C_{+}$,
with $z=\zeta^3$, $\zeta\in\Gamma_N
=\{z\in\C\sothat
0<\arg(\zeta)<\pi/N\}$,
thus finding the expression for the resolvent
$(A+B_0-z I)^{-1}$.
We define $v(\zeta)=R(\zeta)(I-P)f$.
There is the inclusion $f\in L^2_s(\R,\C)\subset L^1_{N-2}(\R,\C)$,
so we can apply $P$ to $f$;
one has
$(I-P)f\in L^1_{N-2,0,\dots,0}(\R,\C)$.
Therefore,
\begin{align}\label{def-vz}
v(\zeta)=R(\zeta)(I-P)f=R_1(\zeta)(I-P)f
\end{align}
is bounded in $L^2_{-s'}(\R,\C)$ uniformly in
$\zeta\in\Gamma_N\cap\mathbb{D}_\delta$
and has a limit as $\zeta\to 0$.
The relation \eqref{def-vz}
gives $(A-z)v(\zeta)=(I-P)f$, hence
\begin{align}
\label{qn3}
(A+B_0-z)v(\zeta)=(I-P)f+B_0 R(\zeta)(I-P)f.
\end{align}
The system \eqref{qn1}, \eqref{qn2}, \eqref{qn3}
contains in the right-hand sides
$\phi_0$, $\phi_1$, and $f$ only;
this
allows us to express $u$ as a linear combination
\[
u=
c_1(\zeta)\psi_1(x,\zeta)
+
c_2(\zeta)\psi_2(x,\zeta)
+
c_3(\zeta)\psi_3(x,\zeta),
\]
with
\[
\psi_1(x,\zeta)
=-\frac{\alpha}{3}+\zeta\varTheta_1(x,\zeta),
\quad
\psi_2
=
\frac{\alpha x}{3}+\zeta\varTheta_0(x,\zeta)
-\jj\alpha\varTheta_1(x,\zeta),
\quad
\psi_3=v(\zeta)=R(\zeta)(I-P)f
\]
and with $c_i(\zeta)$,
$1\le i\le 3$,
uniformly bounded for
$\zeta\in\Gamma_N
\cap\mathbb{D}_\delta$, for $\delta>0$ sufficiently small.

We proved the following result:

\begin{lemma}\label{lemma-main}
Let $A=(-\jj\p_x)^3$, $B_0=\phi_0\otimes\phi_0$,
considered on domain $\dom(A)=H^3(\R,\C)$.
Then the resolvent of the operator
$A+B_0$ satisfies the limiting absorption principle
at the point $z_0=0$
with respect to the triple $L^2_s,\,L^2_{-s'},\,\C_{+}$,
for any $s,\,s'>N-3/2$,
in the sense that
\[
R_{B_0}(z)=(A+B_0-z I)^{-1}:\,L^2_s(\R,\C)\to L^2_{-s'}(\R,\C)
\]
is uniformly bounded for $z\in\C_{+}$
as has a limit (in the uniform operator topology)
as $z\to z_0$.
\end{lemma}

Now we can prove the following theorem:

\begin{theorem}
Let $N=3$ and let $s,\,s'>N-3/2$.
The point $z_0=0$ is a virtual level
of the operator $A=(-\jj\p_x)^3$
with domain $\dom(A)=H^3(\R,\C)$
relative to $L^2_s,\,L^2_{-s'},\,\C_{+}$
(and also
relative to $L^2_s,\,L^2_{-s'},\,\C_{-}$).
\end{theorem}

\begin{proof}
By Lemma~\ref{lemma-main},
the resolvent of the operator
$A+B_0$ satisfies the limiting absorption principle
relative to
$L^2_s,\,L^2_{-s'},\,\varOmega$ at $z_0=0$.
Since $\rank B_0=1$,
by Definition~\ref{def-virtual},
this implies that
$A=(-\jj\p_x)^3$
has at $z_0=0$
a virtual level of rank
\emph{at most} $r=1$
relative to
$L^2_s,\,L^2_{-s'},\,\varOmega$.
To show that
the point $z_0=0$ is a virtual level of
the operator $\p_x^3$
of rank at least one (relative to $L^2_s,\,L^2_{-s'},\,C_{+}$),
we need to show that
the resolvent of $A$ does not satisfy
the limiting absorption principle at $z=0$.
It is enough to notice that
the leading term in \eqref{def-r-kernel}
is given by
$\frac{\jj}{N}
\sum_{j=1}^M
\frac{1}{(\alpha^j\zeta)^{N-1}}$
which
does not have a pointwise limit as $\zeta\to 0$.

Alternatively,
by \cite[Theorem 2.16]{virtual-levels},
it is enough to demonstrate that
there is an arbitrarily small
perturbation $V$
which generates the eigenvalue $z=\zeta^3\ne 0$
near $z_0=0$.
For simplicity,
we drop factors of $\jj$,
considering the equation
\[
-u'''+Vu=z u.
\]
Let $z=\kappa^3$, $\kappa>0$
and define
\[
u_\kappa(x)=
\begin{cases}
e^{-\kappa x},&x\ge 1,
\\
1+\sum\sb{j=0}^7 a_j x^j,
&-1\le x\le 1,
\\
\Re e^{-\alpha \kappa x},
&x\le -1,
\end{cases}
\]
where $\alpha=e^{2\pi \jj/3}$.
We specify $a_j\in\R$, $0\le j\le 7$,
so that
$u_\kappa(x)$ and its three derivatives
are continuous
at $x=\pm 1$;
this leads to $a_j=O(\kappa)$,
$0\le j\le 7$,
hence
$u_\kappa$ converges pointwise to $1$
as $\kappa\to 0$.
We can assume that $\kappa_0>0$ is small enough 
so that for $0\le\kappa<\kappa_0$
one has $u_\kappa\ge 1/2$.
We define $V_\kappa$ by the relation
$\kappa^3 u_\kappa=-u_\kappa'''+V_\kappa u_\kappa$.
Since
$\p_x^3 u_\kappa=-\kappa^3 u_\kappa$ for $\abs{x}\ge 1$,
$V_\kappa\in C(\R,\R)$ is supported on $[-1,1]$;
moreover, $\sup_{x\in[-1,1]}\abs{V_\kappa(x)}\to 0$
as $\kappa\to 0$.
Thus, we have
\[
\kappa^3\in\sigma(-\p_x^3+V_\kappa),
\quad
\kappa>0,
\qquad
V_\kappa\in C_{\mathrm{comp}}([-1,1]),
\qquad
\lim\sb{\kappa\to 0+}
\norm{V_\kappa}_{L^\infty}\to 0.
\]
This produces the family of eigenvalues
bifurcating from $z=0$,
completing the proof.
\end{proof}

Now the main result
(Theorems~\ref{theorem-1}
and~\ref{theorem-2})
follows.
While Theorem~\ref{theorem-1}
just follows from the general theory
developed in \cite{virtual-levels},
let us sketch the proof of
Theorem~\ref{theorem-2}.

\begin{proof}[Proof of Theorem~\ref{theorem-2}]
By Lemma~\ref{lemma-41}
and Corollary~\ref{corollary-no-lap},
existence of a nontrivial solution
$\Psi$ to
\[
\big((-\jj\p_x)^3+V\big)\Psi=0
\]
which grows at most linearly
to $x\to+\infty$
and which is uniformly bounded for
$x\to-\infty$
(which is thus proportional to $\gamma_1(x,0)$)
implies that
there is no limiting absorption principle
relative to
$L^2_\nu,\,L^2_{-\nu},\,\C_{+}$
(and hence
relative to
$L^2_s,\,L^2_{-s'},\,\C_{+}$);
by Lemmata~\ref{lemma-41}
and~\ref{lemma-lap-nu},
this limiting absorption principle is only available
if and only if
$\gamma_1(x,0)$ grows quadratically
as $x\to+\infty$.
If indeed
$\gamma_1$ grows quadratically for $x\to+\infty$,
Lemma~\ref{lemma-main}
gives the limiting absorption principle.

Going in the other direction,
we notice that
absence of the limiting absorption principle
relative to
$L^2_s,\,L^2_{-s'},\,\C_{+}$
implies
-- by \cite{virtual-levels} --
the existence of a nontrivial virtual state
$\Psi\in L^2_{-s'}$
(where we can take
$s'=N-3/2+\epsilon$,
for any $\epsilon>0$
which already shows that $\Psi$
can grow at most linearly at infinity),
which, moreover, belongs to
$\ran((A+B_0-z_0 I)^{-1}_{L^2_s,L^2_{-s'},\varOmega})$
and hence
is bounded for
$x\to-\infty$
as one can see from
\eqref{g-is}.
\end{proof}

\bibliographystyle{sima-doi}
\bibliography{bibcomech}
\end{document}